\documentclass[10pt,letterpaper]{amsart}

\usepackage{amsmath}
\usepackage{amssymb}
\usepackage{amsthm}
\usepackage{amsfonts}
\usepackage{microtype}
\usepackage{mathrsfs}
\usepackage{graphicx}
\usepackage{color}
\usepackage{latexsym}
\usepackage{rotating}
\usepackage{xspace}
\usepackage[all]{xy}
\usepackage{longtable}
\usepackage{leftidx}
\usepackage{mathtools}
\usepackage{titlesec}
\usepackage{setspace}
\usepackage{tikz}
\usepackage{tikz-cd}
\usepackage{url}
\usetikzlibrary{calc}
\usetikzlibrary{arrows}
\usetikzlibrary{decorations.markings,intersections}
\usepackage[top=1.2in, bottom=1.2in]{geometry}
\usepackage{geometry}
\usepackage[final]{pdfpages}
\usepackage{fancyhdr}
\usepackage{multicol}
\usepackage[autostyle]{csquotes}

\newtheorem{theorem}{Theorem}[section]
\numberwithin{equation}{theorem}
\newtheorem*{theorem*}{Theorem}
\newtheorem{lemma}[theorem]{Lemma}

\newtheorem{corollary}[theorem]{Corollary}

\theoremstyle{definition}

\newtheorem{remark}[theorem]{Remark}

\theoremstyle{conjecture}

\newtheorem{mainthm}{Theorem}

\newcommand{\Ass}{\operatorname{Ass}}
\newcommand{\Ker}{\operatorname{Ker}}

\newcommand{\im}{\operatorname{Im}}

\newcommand{\rank}{\operatorname{rank}}

\newcommand{\fd}{\operatorname{fd}}
\newcommand{\id}{\operatorname{id}}
\newcommand{\Gfd}{\operatorname{Gfd}}
\newcommand{\Gid}{\operatorname{Gid}}
\newcommand{\pd}{\operatorname{pd}}
\newcommand{\Gdim}{\operatorname{Gdim}}
\newcommand{\Var}{\operatorname{Var}}

\newcommand{\Ext}{\operatorname{Ext}}

\newcommand{\Tor}{\operatorname{Tor}}
\newcommand{\Hom}{\operatorname{Hom}}

\newcommand{\Soc}{\operatorname{Soc}}
\newcommand{\edim}{\operatorname{edim}}
\newcommand{\depth}{\operatorname{depth}}

\newcommand{\type}{\operatorname{type}}

\newcommand{\cosupp}{\operatorname{cosupp}}
\newcommand{\RHom}{\operatorname{RHom}}
\newcommand{\at}{\operatorname{at}}
\newcommand{\fl}{\operatorname{fl}}
\newcommand{\res}{\operatorname{res}}

\newcommand{\suchthat}{\;\ifnum\currentgrouptype=16 \middle\fi|\;}

\makeatletter
\newcommand{\hocolim@}[2]{%
  \vtop{\m@th\ialign{##\cr
    \hfil$#1\operator@font holim$\hfil\cr
    \noalign{\nointerlineskip\kern1.5\ex@}#2\cr
    \noalign{\nointerlineskip\kern-\ex@}\cr}}%
}
\newcommand{\hocolim}{%
  \mathop{\mathpalette\hocolim@{\rightarrowfill@\textstyle}}\nmlimits@
}
\makeatother

\makeatletter
\newcommand{\holim@}[2]{%
  \vtop{\m@th\ialign{##\cr
    \hfil$#1\operator@font holim$\hfil\cr
    \noalign{\nointerlineskip\kern1.5\ex@}#2\cr
    \noalign{\nointerlineskip\kern-\ex@}\cr}}%
}
\newcommand{\holim}{%
  \mathop{\mathpalette\holim@{\leftarrowfill@\textstyle}}\nmlimits@
}
\makeatother

\makeatletter
\def\@secnumfont{\bfseries}
\def\section{\@startsection{section}{1}%
  \z@{.7\linespacing\@plus\linespacing}{.5\linespacing}%
  {\normalfont\Large\bfseries\filcenter}}
\def\subsection{\@startsection{subsection}{2}%
  \z@{.5\linespacing\@plus.7\linespacing}{-.5em}%
  {\normalfont\large\bfseries}}
\makeatother

\DeclareFontFamily{OT1}{pzc}{}
\DeclareFontShape{OT1}{pzc}{m}{it}{<-> s * [1.20] pzcmi7t}{}
\DeclareMathAlphabet{\mathpzc}{OT1}{pzc}{m}{it}

\makeatletter
\def\moverlay{\mathpalette\mov@rlay}
\def\mov@rlay#1#2{\leavevmode\vtop{%
   \baselineskip\z@skip \lineskiplimit-\maxdimen
   \ialign{\hfil$\m@th#1##$\hfil\cr#2\crcr}}}
\newcommand{\charfusion}[3][\mathord]{
    #1{\ifx#1\mathop\vphantom{#2}\fi
        \mathpalette\mov@rlay{#2\cr#3}
      }
    \ifx#1\mathop\expandafter\displaylimits\fi}
\makeatother

\makeatletter
\providecommand{\bigsqcap}{%
  \mathop{%
    \mathpalette\@updown\bigsqcup
  }%
}
\newcommand*{\@updown}[2]{%
  \rotatebox[origin=c]{180}{$\m@th#1#2$}%
}
\makeatother

\begin{document}

\author[Hossein Faridian]{Hossein Faridian}

\title[Bounds on Gorenstein Dimensions and Exceptional Complete Intersection Maps]
{Bounds on Gorenstein Dimensions and \\ Exceptional Complete Intersection Maps}

\address{Hossein Faridian, School of Mathematical and Statistical Sciences, Clemson University, Clemson, SC 29634, USA.}
\email{hfaridi@g.clemson.edu}



\begin{abstract}
We prove that if $f:R \rightarrow S$ is a local homomorphism of noetherian local rings of finite flat dimension and $M$ is a non-zero finitely generated $S$-module whose Gorenstein flat dimension over $R$ is bounded by the difference of the embedding dimensions of $R$ and $S$, then $M$ is a totally reflexive $S$-module and $f$ is an exceptional complete intersection map. This is an extension of a result of Brochard, Iyengar, and Khare to Gorenstein flat dimension. We also prove two analogues involving Gorenstein injective dimension.
\end{abstract}

\maketitle

\sloppy

\section{Introduction}

An interesting and fruitful theme in commutative algebra is the study of those properties of rings that can be characterized or identified by prescribed conditions on their modules. This has gradually been extended to the relative situation where one investigates certain properties of ring homomorphisms through imposed conditions on the modules along them. An important class of ring homomorphisms is that of complete intersection maps. A surjective homomorphism of noetherian local rings is a complete intersection map if its kernel is generated by a regular sequence. If in addition, the regular sequence is part of a minimal generating set for the maximal ideal of the source, then the homomorphism is an \enquote{exceptional} complete intersection map. This subclass is of particular interest as it appears in various contexts; for example, the diagonal of a smooth map is locally exceptional complete intersection. Recently, some research has been conducted to characterize or detect this property.

In one direction, Iyengar, Letz, Liu, and Pollitz characterize surjective exceptional complete intersection maps in terms of lattices of thick subcategories of bounded derived categories as well as the vanishing of certain Atiyah classes. More precisely, they prove that if $f:(R,\mathfrak{m},k) \rightarrow (S,\mathfrak{n},k)$ is a surjective local homomorphism of noetherian local rings, then $f$ is an exceptional complete intersection map if and only if $\fd_{R}(S) < \infty$ and the restriction functor $\res_{f}:\mathcal{D}_{\square}^{\fl}(S)\rightarrow \mathcal{D}_{\square}^{\fl}(R)$ induces an isomorphism between the corresponding lattices of thick subcategories. Furthermore, they show that such an $f$ is an exceptional complete intersection map if and only if $\fd_{R}(S) < \infty$ and $\at^{f}(k)=0$; see \cite[Theorems A and B]{ILP}.

In another direction, Brochard, Iyengar, and Khare detect exceptional complete intersection maps through a certain bound on the flat dimension of a module along them. More specifically, they prove that if $f:R \rightarrow S$ is a local homomorphism of noetherian local rings, and $M$ is a non-zero finitely generated $S$-module with $\fd_{R}(M)\leq \edim(R)- \edim(S)$, then $M$ is a free $S$-module and $f$ is an exceptional complete intersection map; see \cite[Theorem 3.1]{BIK}. Dually, the author shows that if $f:R \rightarrow S$ is a local homomorphism of noetherian local rings, and $M$ is a non-zero finitely generated or artinian $S$-module with $\id_{R}(M)\leq \edim(R)- \edim(S)$, then $M$ is an injective $S$-module and $f$ is an exceptional complete intersection map; see \cite[Theorems 3.3 and 3.12]{Fa}.

The goal of this article is to establish some analogues of the results mentioned in the previous paragraph for Gorenstein dimensions. In other words, we show that under some mild assumptions, the same bound on Gorenstein dimensions of a module along a map can detect the exceptional complete intersection property. More specifically, we prove the following; see Theorem \ref{3.5}.

\begin{mainthm}
Let $f:(R,\mathfrak{m}) \rightarrow (S,\mathfrak{n})$ be a local homomorphism of noetherian local rings with $\fd_{R}(S)< \infty$, and $M$ a non-zero finitely generated $S$-module with $\Gfd_{R}(M)\leq \edim(R)- \edim(S)$. Then the following assertions hold:
\begin{enumerate}
\item[(i)] $M$ is a totally reflexive $S$-module.
\item[(ii)] $f$ is an exceptional complete intersection map.
\item[(iii)] $\Gfd_{R}(M) = \edim(R)- \edim(S)$.
\end{enumerate}
\end{mainthm}

One notes that Gorenstein flat dimension does not exceed flat dimension, so the above result could be conceived as an extension of \cite[Theorem 3.1]{BIK} to Gorenstein flat dimension. We further discuss how the finite flat dimension in the previous theorem might be replaced by a G-regularity assumption; see Remark \ref{3.9}.

Dually, We prove two analogues for Gorenstein injective dimension; see Theorems \ref{3.11} and \ref{3.13}.

\begin{mainthm}
Let $f:(R,\mathfrak{m},k) \rightarrow (S,\mathfrak{n},l)$ be a local homomorphism of noetherian local rings with $\fd_{R}(S)<\infty$ in which $R$ has a dualizing complex and $S$ is $\mathfrak{n}$-adically complete, and $M$ a non-zero artinian $S$-module with $\Gid_{R}(M)\leq \edim(R)- \edim(S)$. Then the following assertions hold:
\begin{enumerate}
\item[(i)] $M$ is a Gorenstein injective $S$-module.
\item[(ii)] $f$ is an exceptional complete intersection map.
\item[(iii)] $\Gid_{R}(M)= \edim(R)- \edim(S)$.
\end{enumerate}
\end{mainthm}

\begin{mainthm}
Let $f:(R,\mathfrak{m},k) \rightarrow (S,\mathfrak{n},l)$ be a local homomorphism of noetherian local rings with $\fd_{R}(S)<\infty$ in which $R$ has a dualizing complex and $S$ is G-regular, and $M$ a non-zero artinian $S$-module with $\Gid_{R}(M)\leq \edim(R)- \edim(S)$. Then the following assertions hold:
\begin{enumerate}
\item[(i)] $M$ is an injective $S$-module.
\item[(ii)] $f$ is an exceptional complete intersection map.
\item[(iii)] $\Gid_{R}(M)= \edim(R)- \edim(S)$.
\end{enumerate}
\end{mainthm}

\section{Basic Definitions and Observations}

In what follow, all rings are assumed to be non-zero commutative with unity. We write $(R,\mathfrak{m},k)$ to indicate $R$ is a local ring with the maximal ideal $\mathfrak{m}$ and the residue field $k\cong R/ \mathfrak{m}$. Moreover, $\widehat{R}^{\mathfrak{m}}$ denotes its $\mathfrak{m}$-adic completion. A ring homomorphism $f:(R,\mathfrak{m}) \rightarrow (S,\mathfrak{n})$ is local if $f(\mathfrak{m})\subseteq \mathfrak{n}$. Given an $R$-module $M$, the notations $\pd_{R}(M)$, $\id_{R}(M)$, and $\fd_{R}(M)$ are used for projective, injective, and flat dimensions of $M$, respectively. Moreover, $\depth_{R}(\mathfrak{m},M)$, $\dim(R)$, and $\edim(R)$ indicate depth, Krull dimension, and embedding dimension, respectively.

We begin with a remark on totally reflexive modules and G-regular rings.

\begin{remark} \label{2.1}
Let $R$ be a ring, and consider the functor $(-)^{\ast} = \Hom_{R}(-,R)$. A finitely generated $R$-module $G$ is said to be \textit{totally reflexive} if the biduality map $\eta_{G}: G\rightarrow G^{\ast\ast}$, given by $\eta_{G}(x)(f)=f(x)$ for every $x\in G$ and $f\in G^{\ast}$, is an isomorphism, and $\Ext_{R}^{i}(G,R)=0=\Ext_{R}^{i}(G^{\ast},R)$ for every $i \geq 1$; see \cite[Definition 1.1.2]{Ch}. Given a non-zero finitely generated $R$-module $M$, the \textit{G-dimension} of $M$, denoted by $\Gdim_{R}(M)$, is the least integer $n\geq 0$ for which there exists an exact sequence
$$0 \rightarrow G_{n} \rightarrow \cdots \rightarrow G_{1} \rightarrow G_{0} \rightarrow M \rightarrow 0$$
where $G_{i}$ is a totally reflexive $R$-module for every $0\leq i \leq n$. We have $\Gdim_{R}(M)\leq \pd_{R}(M)$ with equality if $\pd_{R}(M)<\infty$; see \cite[Proposition 1.2.10]{Ch}.

Every finitely generated free $R$-module is clearly totally reflexive. Following Takahashi, \cite[Definition 1.7]{Ta}, we say that a noetherian local ring $R$ is \textit{G-regular} if the converse holds, i.e. every totally reflexive $R$-module is free. If $R$ is G-regular, then by \cite[Proposition 1.8 (2)]{Ta} coupled with \cite[Theorem 8.27 (i)]{Ro}, we have $\Gdim_{R}(M)=\pd_{R}(M)=\fd_{R}(M)$ for every finitely generated $R$-module $M$. The class of G-regular rings is quite large:
\begin{enumerate}
\item[(i)] Every regular local ring is G-regular. Indeed, if $R$ is a regular local ring and $G$ is a totally reflexive $R$-module, then $\pd_{R}(G)< \infty$, so $\pd_{R}(G)= \Gdim_{R}(G)=0$, whence $G$ is free as $R$ is local.
\item[(ii)] Every Golod local ring which is not a hypersurface is G-regular; see \cite[Example 3.5 (2)]{AM}. In particular, every non-Gorenstein Cohen-Macaulay local ring with minimal multiplicity is G-regular; see \cite[Lemma 5.1]{Ta}. As for concrete examples of this kind, if $k$ is a field, then $k[[X,Y]]/(X^{2},XY,Y^{2})$, $k[[X,Y,Z]]/(X^{2}-YZ,Y^{2}-XZ,Z^{2}-XY)$, and $k[[X^{3},X^{4},X^{5}]]$ are G-regular; see \cite[Example 5.2]{Ta}.
\item[(iii)] If $k$ is field, then $k[[X,Y]]/(X^{3},XY,Y^{3})$ is an artinian non-Gorenstein G-regular local ring which does not have minimal multiplicity; see \cite[Examples 5.4 (1)]{Ta}. Also, $k[[X,Y,Z]]/(X^{3}-Y^{2}Z,Y^{3}-X^{2}Z,Z^{2}-XY)$ is a one-dimensional Cohen-Macaulay non-Gorenstein G-regular local ring which does not have minimal multiplicity; see \cite[Examples 5.4 (2)]{Ta}.
\item[(iv)] Given noetherian local rings $(R,\mathfrak{m},k)$ and $(S,\mathfrak{n},k)$ with a common residue field $k$, if the pullback
    \begin{equation*}
    \begin{tikzcd}
    R\times_{k}S \arrow{r} \arrow{d} & S \arrow{d}
    \\
    R \arrow{r} & k
    \end{tikzcd}
    \end{equation*}
    is not Gorenstein, then it is G-regular; see \cite[Corollary 4.7]{NS}. On the other hand, by \cite[Fact 2.2]{NTSV}, the ring $R\times_{k}S$ is Cohen-Macaulay if and only if $R$ and $S$ are Cohen-Macaulay with $\dim(R) = \dim(S) \leq 1$. This indicates that there are plenty of examples of non-Cohen-Macaulay G-regular rings; see also \cite[Remark 2.9 (ii)]{DMT}. As for concrete examples, if $k$ is a field, then $k[[X,Y]]/(X^{2},XY)$ is a non-Cohen-Macaulay G-regular local ring; see \cite[Examples 5.5]{Ta}.
\item[(v)] Every non-Gorenstein quotient of small colength of a deeply embedded equicharacteristic artinian Gorenstein local ring is G-regular; see \cite{KV}.
\item[(vi)] If $(R,\mathfrak{m})$ is a G-regular local ring, then the $\mathfrak{m}$-adic completion $\widehat{R}^{\mathfrak{m}}$, the quotient $R/(a)$ by an $R$-regular element $a\in \mathfrak{m} \setminus \mathfrak{m}^{2}$, and the power series ring $R[[X_{1},...,X_{n}]]$ for every $n\geq 1$, are G-regular; see \cite[Corollary 4.7, Proposition 4.6, and Corollary 4.4]{Ta}.
\end{enumerate}
\end{remark}

We next recall Gorenstein dimensions which are refinements of classical homological dimensions.

\begin{remark} \label{2.2}
Let $R$ be a ring, and $\mathcal{I}$ denote the class of injective $R$-modules. An $R$-module $Q$ is \textit{Gorenstein flat} if there exists an $(\mathcal{I}\otimes_{R}-)$-exact exact $R$-complex
$$F: \cdots \rightarrow F_{2} \xrightarrow {\partial^{F}_{2}} F_{1} \xrightarrow {\partial^{F}_{1}} F_{0} \xrightarrow {\partial^{F}_{0}} F_{-1} \xrightarrow {\partial^{F}_{-1}} F_{-2} \rightarrow \cdots$$
of flat modules such that $Q \cong \im \partial^{F}_{0}$; see \cite[Definition 10.3.1]{EJ}. Given a non-zero $R$-module $M$, the \textit{Gorenstein flat dimension} of $M$, denoted by $\Gfd_{R}(M)$, is the least integer $n\geq 0$ for which there exists an exact sequence
$$0 \rightarrow Q_{n} \rightarrow \cdots \rightarrow Q_{1} \rightarrow Q_{0} \rightarrow M \rightarrow 0$$
where $Q_{i}$ is a Gorenstein flat $R$-module for every $0\leq i \leq n$. We have $\Gfd_{R}(M)\leq \fd_{R}(M)$ with equality if $\fd_{R}(M)<\infty$; see \cite[Propositions 4.8]{CFH1}. This shows that Gorenstein flat dimension is finer than flat dimension. Moreover, if $R$ is noetherian and $M$ is finitely generated, then $\Gfd_{R}(M)= \Gdim_{R}(M)$; see \cite[Proposition 4.24]{CFH1}.

Dually, an $R$-module $J$ is \textit{Gorenstein injective} if there exists a $\Hom_{R}(\mathcal{I},-)$-exact exact $R$-complex
$$I: \cdots \rightarrow I_{2} \xrightarrow {\partial^{I}_{2}} I_{1} \xrightarrow {\partial^{I}_{1}} I_{0} \xrightarrow {\partial^{I}_{0}} I_{-1} \xrightarrow {\partial^{I}_{-1}} I_{-2} \rightarrow \cdots$$
of injective modules such that $J \cong \im \partial^{I}_{0}$; see \cite[Definition 10.1.1]{EJ}. Given a non-zero $R$-module $M$, the \textit{Gorenstein injective dimension} of $M$, denoted by $\Gid_{R}(M)$, is the least integer $n\geq 0$ for which there exists an exact sequence
$$0 \rightarrow M \rightarrow J_{0} \rightarrow J_{-1} \rightarrow \cdots \rightarrow J_{-n} \rightarrow 0$$
where $J_{-i}$ is a Gorenstein injective $R$-module for every $0\leq i \leq n$. We have $\Gid_{R}(M)\leq \id_{R}(M)$ with equality if $\id_{R}(M)<\infty$; see \cite[Propositions 3.10]{CFH1}. This shows that Gorenstein injective dimension is finer than injective dimension.
\end{remark}

We finally recall the concept of an (exceptional) complete intersection map.

\begin{remark} \label{2.3}
Let $f:(R,\mathfrak{m},k) \rightarrow (S,\mathfrak{n},l)$ be a local homomorphism of noetherian local rings. Then $f$ admits a Cohen factorization, i.e. it fits into a commutative diagram
\begin{equation*}
  \begin{tikzcd}
  R \arrow{r}{f} \arrow{d}[swap]{\dot{f}} & S \arrow{d}
  \\
  R' \arrow{r}{f'} & \widehat{S}^{\mathfrak{n}}
\end{tikzcd}
\end{equation*}

\noindent
of noetherian local rings in which $\dot{f}$ is flat with regular closed fiber $R'/ \mathfrak{m}R'$, $R'$ is complete with respect to its maximal ideal, and $f'$ is surjective; see \cite[Theorem 1.1]{AFHe}. Accordingly, $f$ is said to be a \textit{complete intersection map} if there is a Cohen factorization of $f$ in which $\Ker(f')$ is generated by a regular sequence on $R'$. This property is independent of the choice of Cohen factorization; see \cite[Theorem 1.2]{Av}. In particular, $f$ is complete intersection if and only if $f'$ is complete intersection. Moreover, the class of complete intersection maps is closed under composition and flat base change.

When $f$ is complete intersection, there is an inequality
$$\edim(R)-\dim(R) \leq \edim(S)-\dim(S);$$
see \cite[Lemma 2.13]{BIK}. Then $f$ is said to be an \textit{exceptional complete intersection map} if equality holds above. When $f$ is surjective, this property is equivalent to the condition that $\Ker(f)$ can be generated by a regular sequence whose image in $\mathfrak{m}/ \mathfrak{m}^{2}$ is a linearly independent set over $k$; see \cite[Lemma 3.2]{ILP}. Considering a Cohen factorization of $f$ as above and applying \cite[2.6]{BIK} to $\dot{f}$, we obtain
$$\edim(R')- \edim(R) = \edim\left(R'/ \mathfrak{m}R'\right) = \dim(R')- \dim(R),$$
so $\edim(R')- \dim(R') = \edim(R) - \dim(R)$. It follows that $f$ is exceptional complete intersection if and only if $f'$ is exceptional complete intersection.

Exceptional complete intersection maps are prevalent. Here are a few examples:
\begin{enumerate}
\item[(i)] If $f:(R,\mathfrak{m})\rightarrow (S,\mathfrak{n})$ is a flat local homomorphism of noetherian local rings whose closed fiber $S/ \mathfrak{m}S$ is regular, then $f$ is an exceptional complete intersection map; see \cite[2.14]{BIK}.
\item[(ii)] If $f:R\rightarrow S$ is a local homomorphism of regular local rings, then $f$ is an exceptional complete intersection map; see \cite[Example 2.15]{BIK}.
\item[(ii)] If $f:R \rightarrow S$ is a flat local homomorphism of noetherian local rings which is essentially of finite type and smooth, then the diagonal map $\mu:S\otimes_{R}S \rightarrow S$, given by $\mu(a\otimes b)=ab$ for every $a,b\in S$, is an exceptional complete intersection map; see \cite[Theorem 9.9]{Iy1}.
\end{enumerate}

For more information on (exceptional) complete intersection maps, refer to \cite{Av}, \cite{BIK}, \cite{ILP}, \cite{BILP}, and \cite{Iy1}.
\end{remark}

\section{Main Results}

In this section, we prove our main results. The following lemma generalizes the Auslander-Buchsbaum formula to Gorenstein flat dimension.

\begin{lemma} \label{3.1}
Let $(R,\mathfrak{m})$ be a noetherian local ring, and $M$ a non-zero finitely generated or $\mathfrak{m}$-adically complete $R$-module with $\Gfd_{R}(M)< \infty$. Then we have:
$$\Gfd_{R}(M)= \depth(\mathfrak{m},R) - \depth_{R}(\mathfrak{m},M)$$
\end{lemma}

\begin{proof}
If $M$ is finitely generated, then by remark \ref{2.2}, we have $\Gfd_{R}(M)=\Gdim_{R}(M)$, so the result is \cite[Theorem 1.25]{CFH1}. Now assume that $M$ is $\mathfrak{m}$-adically complete. Then \cite[Theorem 13.1.30]{CFH2} implies that $M \simeq \textrm{L}\Lambda^{\mathfrak{m}}(M)$ in the derive category $\mathcal{D}(R)$, so by \cite[Theorem 14.4.18]{CFH2}, $\cosupp_{R}(M)\subseteq \Var(\mathfrak{m})=\{\mathfrak{m}\}$. However, $M\neq 0$, so \cite[Theorem 14.3.35]{CFH2} yields $\cosupp_{R}(M)\neq \emptyset$. It follows that $\cosupp_{R}(M)=\{\mathfrak{m}\}$. Therefore, we conclude from \cite[Theorem 19.4.9]{CFH2} that:
\begin{equation*}
\begin{split}
 \Gfd_{R}(M) & = \sup \left\{\depth\left(\mathfrak{p}R_{\mathfrak{p}},R_{\mathfrak{p}}\right) - \depth_{R_{\mathfrak{p}}}\left(\mathfrak{p}R_{\mathfrak{p}},\RHom_{R}\left(R_{\mathfrak{p}},M\right)\right) \suchthat \mathfrak{p}\in \cosupp_{R}(M)\right\} \\
 & = \depth\left(\mathfrak{m}R_{\mathfrak{m}},R_{\mathfrak{m}}\right) - \depth_{R_{\mathfrak{m}}}\left(\mathfrak{m}R_{\mathfrak{m}},\RHom_{R}\left(R_{\mathfrak{m}},M\right)\right) \\
 & = \depth(\mathfrak{m},R) - \depth_{R}\left(\mathfrak{m},\RHom_{R}\left(R,M\right)\right) \\
 & = \depth(\mathfrak{m},R) - \depth_{R}\left(\mathfrak{m},M\right)
\end{split}
\end{equation*}
\end{proof}

The next lemma can be viewed as a generalization of \cite[Lemma 2.7]{AFHa}.

\begin{lemma} \label{3.2}
Let $f:(R,\mathfrak{m}) \rightarrow (S,\mathfrak{n})$ be a flat local homomorphism of complete noetherian local rings with Gorenstein closed fiber $S/ \mathfrak{m}S$, and $M$ a finitely generated $S$-module. Then we have:
$$\Gfd_{S}(M)\leq \Gfd_{R}(M) + \depth\left(\mathfrak{n}/ \mathfrak{m}S,S/ \mathfrak{m}S\right)$$
\end{lemma}

\begin{proof}
If $\Gfd_{R}(M)=\infty$, then there is nothing to prove, so we may assume that $\Gfd_{R}(M)< \infty$. Since $R$ is $\mathfrak{m}$-adically complete, Cohen's Structure Theorem implies that $R$ is a homomorphic image of a regular hence Gorenstein local ring, thereby $R$ has a dualizing complex, say $D$; see \cite{Ka}. As $\Gfd_{R}(M)< \infty$, \cite[Theorem 9.2]{CFH1} implies that $M\in \mathcal{A}_{D}(R)$ where $\mathcal{A}_{D}(R)$ is the Auslander class of $R$ with respect to $D$. Since $f$ is flat and $S/ \mathfrak{m}S$ is Gorenstein, we conclude that $f$ is Gorenstein, so $D\otimes_{R}S$ is a dualizing complex for $S$ and $M\in \mathcal{A}_{D\otimes_{R}S}(S)$; see \cite[Paragraph before (2.11) and Proposition 3.7 (b)]{AF2}. Therefore, \cite[Theorem 9.2]{CFH1} yields $\Gfd_{S}(M)< \infty$. Now Lemma \ref{3.1} gives:
$$\Gfd_{S}(M)= \depth(\mathfrak{n},S) - \depth_{S}(\mathfrak{n},M)$$
On the other hand, $S$ is $\mathfrak{n}$-adically complete and $M$ is a finitely generated $S$-module, so we conclude that $M\cong M\otimes_{S}S \cong M\otimes_{S}\widehat{S}^{\mathfrak{n}} \cong \widehat{M}^{\mathfrak{n}}$, whence $M$ is an $\mathfrak{n}$-adically complete $S$-module. But $f$ is local, so $\mathfrak{m}S \subseteq \mathfrak{n}$, whence \cite[2.2.9 (b)]{SS} implies that $M$ is $\mathfrak{m}S$-adically complete. However, $(\mathfrak{m}S)^{n}M=\mathfrak{m}^{n}SM=\mathfrak{m}^{n}M$ for every $n\geq 1$, so we notice that:
$$M \cong \widehat{M}^{\mathfrak{m}S} = \underset{n\geq 1}\varprojlim M/(\mathfrak{m}S)^{n}M = \underset{n\geq 1}\varprojlim M/\mathfrak{m}^{n}M = \widehat{M}^{\mathfrak{m}}$$
As $\mathfrak{m}$ is finitely generated, \cite[Theorem 2.2.2]{SS} implies that $\widehat{M}^{\mathfrak{m}}$ is an $\mathfrak{m}$-adically complete $R$-module, so the above display shows that $M$ is an $\mathfrak{m}$-adically complete $R$-module. Therefore, Lemma \ref{3.1} gives:
$$\Gfd_{R}(M)= \depth(\mathfrak{m},R) - \depth_{R}(\mathfrak{m},M)$$
Moreover, \cite[Proposition 1.2.16]{BH} implies that:
$$\depth(\mathfrak{n},S) = \depth(\mathfrak{m},R) + \depth\left(\mathfrak{n}/ \mathfrak{m}S,S/ \mathfrak{m}S\right)$$
Also, $\depth_{R}(\mathfrak{m},M) \leq \depth_{S}(\mathfrak{n},M)$ by \cite[Proposition 5.2 (2)]{Iy2}. Putting everything together, we get:
\begin{equation*}
\begin{split}
 \Gfd_{S}(M) & = \depth(\mathfrak{n},S) - \depth_{S}(\mathfrak{n},M) \\
 & \leq \depth(\mathfrak{n},S) - \depth_{R}(\mathfrak{m},M) \\
 & = \depth(\mathfrak{m},R) - \depth_{R}(\mathfrak{m},M) + \depth\left(\mathfrak{n}/ \mathfrak{m}S,S/ \mathfrak{m}S\right) \\
 & = \Gfd_{R}(M) + \depth\left(\mathfrak{n}/ \mathfrak{m}S,S/ \mathfrak{m}S\right)
\end{split}
\end{equation*}
\end{proof}

\begin{corollary} \label{3.3}
Let $f:(R,\mathfrak{m}) \rightarrow (S,\mathfrak{n})$ be a flat local homomorphism of complete noetherian local rings with regular closed fiber $S/ \mathfrak{m}S$, and $M$ a finitely generated $S$-module. Then we have:
$$\Gfd_{S}(M)\leq \Gfd_{R}(M) + \edim(S)- \edim(R)$$
\end{corollary}

\begin{proof}
The closed fiber $S/ \mathfrak{m}S$ is regular, hence Gorenstein, so Lemma \ref{3.2} together with \cite[2.6]{BIK} yield:
\begin{equation*}
\begin{split}
 \Gfd_{S}(M) & \leq \Gfd_{R}(M) + \depth\left(\mathfrak{n}/ \mathfrak{m}S,S/ \mathfrak{m}S\right) \\
 & = \Gfd_{R}(M) + \edim\left(S/ \mathfrak{m}S\right) \\
 & = \Gfd_{R}(M) + \edim(S)- \edim(R)
\end{split}
\end{equation*}
\end{proof}

\begin{lemma} \label{3.4}
Let $f:(R,\mathfrak{m})\rightarrow (S,\mathfrak{n})$ be an epimorphism of noetherian local rings with $\fd_{R}(S)< \infty$. If $f$ is not an isomorphism, then $\Ker(f)$ contains a regular element on $R$.
\end{lemma}

\begin{proof}
Since $f$ is surjective, we have $\pd_{R}(S)= \fd_{R}(S)< \infty$. In addition, $\Ker(f)\neq 0$, so \cite[Corollary 6.3]{AB} yields a non-zerodivisor $a\in \Ker(f)$. But $\Ker(f)\subseteq \mathfrak{m}$, so $a$ is regular on $R$.
\end{proof}

We are now ready to prove our first main result which is Theorem A from the introduction.

\begin{theorem} \label{3.5}
Let $f:(R,\mathfrak{m}) \rightarrow (S,\mathfrak{n})$ be a local homomorphism of noetherian local rings with $\fd_{R}(S)< \infty$, and $M$ a non-zero finitely generated $S$-module with $\Gfd_{R}(M)\leq \edim(R)- \edim(S)$. Then the following assertions hold:
\begin{enumerate}
\item[(i)] $M$ is a totally reflexive $S$-module.
\item[(ii)] $f$ is an exceptional complete intersection map.
\item[(iii)] $\Gfd_{R}(M) = \edim(R)- \edim(S)$.
\end{enumerate}
In particular, if $M$ is Gorenstein flat as an $R$-module, then $\edim(R)= \edim(S)$.
\end{theorem}

\begin{proof}
We first reduce to the case where $R$ is $\mathfrak{m}$-adically complete and $S$ is $\mathfrak{n}$-adically complete. Note that $M\otimes_{S}\widehat{S}^{\mathfrak{n}}$ is a non-zero finitely generated $\widehat{S}^{\mathfrak{n}}$-module. Since $\Gfd_{R}(M)< \infty$, we can invoke \cite[Corollary 4.8]{CI} to observe that:
\begin{equation*}
\begin{split}
 \Gfd_{\widehat{R}^{\mathfrak{m}}}\left(M\otimes_{S}\widehat{S}^{\mathfrak{n}}\right) & = \Gfd_{R}(M) \\
 & \leq \edim(R)- \edim(S) \\
 & = \edim\left(\widehat{R}^{\mathfrak{m}}\right)- \edim\left(\widehat{S}^{\mathfrak{n}}\right)
\end{split}
\end{equation*}
On the other hand, \cite[Proposition 1.26]{CFH1} gives $\Gdim_{S}(M) = \Gdim_{\widehat{S}^{\mathfrak{n}}}\left(M\otimes_{S}\widehat{S}^{\mathfrak{n}}\right)$, so it follows that $M$ is a totally reflexive $S$-module if and only if $M\otimes_{S}\widehat{S}^{\mathfrak{n}}$ is a totally reflexive $\widehat{S}^{\mathfrak{n}}$-module. In view of Remark \ref{2.3}, consider the commutative diagram
\begin{equation*}
  \begin{tikzcd}
  R \arrow{r}{f} \arrow{d} & S \arrow{d}
  \\
  \widehat{R}^{\mathfrak{m}} \arrow{r}{\widehat{f}} \arrow{d}[swap]{\dot{f}} & \widehat{S}^{\mathfrak{n}} \ar[equal]{d}
  \\
  R' \arrow{r}{f'} & \widehat{S}^{\mathfrak{n}}
\end{tikzcd}
\end{equation*}
in which the bottom square provides a Cohen factorization of $\widehat{f}$. Then it is clear that the combined rectangle provides a Cohen factorization of $f$. As a result, $f$ is exceptional complete intersection if and only if $f'$ is exceptional complete intersection if and only if $\widehat{f}$ is exceptional complete intersection. Finally, since the combined rectangle gives a Cohen factorization of $f$, and $\fd_{R}(S)< \infty$, \cite[3.3]{AFHe} implies that $\fd_{R'}\left(\widehat{S}^{\mathfrak{n}}\right)< \infty$. But then \cite[Corollary 4.2(b)(F)]{AF1} yields
$$\fd_{\widehat{R}^{\mathfrak{m}}}\left(\widehat{S}^{\mathfrak{n}}\right)\leq \fd_{R'}\left(\widehat{S}^{\mathfrak{n}}\right) + \fd_{\widehat{R}^{\mathfrak{m}}}(R') = \fd_{R'}\left(\widehat{S}^{\mathfrak{n}}\right) < \infty.$$
All in all, we can replace $f$ with $\widehat{f}$, $M$ with $M\otimes_{S}\widehat{S}^{\mathfrak{n}}$, and assume accordingly that $R$ is $\mathfrak{m}$-adically complete and $S$ is $\mathfrak{n}$-adically complete.

We next reduce to the case where $f$ is surjective. By Remark \ref{2.3}, $f$ has a Cohen factorization
$$R \xrightarrow{\dot{f}} R' \xrightarrow{f'} S$$
in which $\dot{f}$ is flat with regular closed fiber $R'/ \mathfrak{m}R'$, $R'$ is complete, and $f'$ is surjective. Since $M$ is a finitely generated $S$-module, we conclude that $M$ is a finitely generated $R'$-module as well. Moreover, $f$ is exceptional complete intersection if and only if $f'$ is exceptional complete intersection. Using Corollary \ref{3.3} and the hypothesis, we obtain:
\begin{equation*}
\begin{split}
 \Gfd_{R'}(M) & \leq \Gfd_{R}(M) + \edim(R')-\edim(R) \\
 & \leq \edim(R)-\edim(S) + \edim(R')-\edim(R) \\
 & = \edim(R')-\edim(S)
\end{split}
\end{equation*}
If $\Gfd_{R'}(M)=\edim(R')-\edim(S)$, then the above display shows that $\Gfd_{R}(M)=\edim(R)-\edim(S)$. Finally, since $\fd_{R}(S)< \infty$, \cite[3.3]{AFHe} implies that $\fd_{R'}(S)< \infty$. Hence it suffices to prove the result for $f'$. Therefore, we can assume that $f$ is surjective, so that $M$ is a finitely generated $R$-module as well. Now by Remark \ref{2.2}, $\Gfd_{R}(M) = \Gdim_{R}(M)$.

We finally reduce to the case where $\Ker(f)\subseteq \mathfrak{m}^{2}$. Assume to the contrary that $\Ker(f)\nsubseteq \mathfrak{m}^{2}$. In particular, $f$ is not an isomorphism, so Lemma \ref{3.4} implies that $\Ker(f)$ contains a regular element on $R$, thereby $\Ker(f)\nsubseteq \mathcal{Z}(R)=\bigcup_{\mathfrak{p}\in \Ass(R)}\mathfrak{p}$. By the Prime Avoidance Lemma (see \cite[Theorem 3.61]{Sh}), we infer that $\Ker(f)\nsubseteq \mathfrak{m}^{2}\cup \bigcup_{\mathfrak{p}\in \Ass(R)}\mathfrak{p}$, so there exists a regular element $a\in \Ker(f)\backslash \mathfrak{m}^{2}$. We then get the factorization
$$R\rightarrow R/(a) \xrightarrow{f'} S$$
of $f$. We note that $\Ker(f)$ is generated by a regular sequence if and only if $\Ker(f')$ is generated by a regular sequence. Hence $f$ is complete intersection if and only $f'$ is complete intersection. In addition, $\edim\left(R/(a)\right)= \edim(R)-1$ and $\dim\left(R/(a)\right)= \dim(R)-1$, so $\edim\left(R/(a)\right)-\dim\left(R/(a)\right)= \edim(R)- \dim(R)$. It follows that $f$ is exceptional complete intersection if and only $f'$ is exceptional complete intersection. On the other hand, $\Gdim_{R/(a)}(M)= \Gdim_{R}(M)-1$ by \cite[Theorem 4.1]{BM}, so we get
$\Gdim_{R/(a)}(M)\leq \edim\left(R/(a)\right)-\edim(S)$. Finally, $a$ is a regular element on $R$ that is not contained in $\mathfrak{m}^{2}$, so \cite[Corollary 27.5]{Na} yields $\fd_{R/(a)}(S)= \pd_{R/(a)}(S)= \pd_{R}(S)-1 < \infty$. As a result, we may replace $f$ by $f'$, and continue this process. Since $\Ker(f)$ cannot contain any infinite regular sequence, we cannot continue this process indefinitely, so at some finite step, the kernel of the resulting map will be contained in $\mathfrak{m}^{2}$. Therefore, we can assume that $\Ker(f)\subseteq \mathfrak{m}^{2}$.

With these reductions, we end up with a local epimorphism $f:R\rightarrow S$ of noetherian local rings with $\fd_{R}(S)< \infty$ and $\Ker(f)\subseteq \mathfrak{m}^{2}$, and a finitely generated $S$-module $M$ such that $\Gdim_{R}(M)\leq \edim(R)-\edim(S)$. But then the assumption $\Ker(f)\subseteq \mathfrak{m}^{2}$ implies that $\edim(R)=\edim(S)$, so $\Gdim_{R}(M)=0$. If $f$ is not an isomorphism, then as in the previous paragraph, we can find a regular element $a\in \Ker(f)$ so that $\Gdim_{R/(a)}(M)= \Gdim_{R}(M)-1= -1$ which is a contradiction. Therefore, $f$ is an isomorphism, so in particular, it is an exceptional complete intersection map. Furthermore, $\Gdim_{S}(M)=\Gdim_{R}(M)=0=\edim(R)-\edim(S)$, so $M$ is a totally reflexive $S$-module, and the equality in (iii) trivially holds.
\end{proof}

\begin{corollary} \label{3.6}
Let $f:(R,\mathfrak{m}) \rightarrow (S,\mathfrak{n})$ be a local homomorphism of noetherian local rings with $\fd_{R}(S)< \infty$, and $M$ a non-zero finitely generated $S$-module. Then we have
$$\Gfd_{R}(M)\geq \edim(R)- \edim(S),$$
and if equality holds above, then $f$ is an exceptional complete intersection map.
\end{corollary}

\begin{proof}
We note that if equality holds above, then by Theorem \ref{3.5}, $f$ is an exceptional complete intersection map. Thus we are left with establishing the inequality. If $\Gfd_{R}(M)= \infty$, then there is nothing to prove, so we may assume that $\Gfd_{R}(M)< \infty$.

We first handle the special case where $f$ is surjective. Then $M$ is a finitely generated $R$-module as well, so by Remark \ref{2.2}, $\Gfd_{R}(M)= \Gdim_{R}(M)$. We argue by induction on $\edim(R) - \edim(S)$. If $\edim(R) - \edim(S)=0$, then there is nothing to prove. Now suppose that $\edim(R) - \edim(S)\geq 1$. Then $f$ is not an isomorphism, so Lemma \ref{3.4} implies that $\Ker(f)$ contains a regular element on $R$, so $\Ker(f)\nsubseteq \mathcal{Z}(R)=\bigcup_{\mathfrak{p}\in \Ass(R)}\mathfrak{p}$. On the other hand, $\edim(S)=\edim\left(R/ \Ker(f)\right)\leq \edim(R)$ with equality if and only if $\Ker(f)\subseteq \mathfrak{m}^{2}$. But $\edim(R)-\edim(S)\geq 1$, so $\Ker(f)\nsubseteq \mathfrak{m}^{2}$. By the Prime Avoidance Lemma (see \cite[Theorem 3.61]{Sh}), we infer that $\Ker(f)\nsubseteq \mathfrak{m}^{2}\cup \bigcup_{\mathfrak{p}\in \Ass(R)}\mathfrak{p}$, so there exists a regular element $a\in \Ker(f)\backslash \mathfrak{m}^{2}$. As in the proof of Theorem \ref{3.5}, we notice that $\fd_{R/(a)}(S)< \infty$. Thus in light of \cite[Theorem 4.1]{BM}) and the induction hypothesis applied to $\bar{f}:R/ \langle a \rangle \rightarrow S$, we see that:
\begin{equation*}
\begin{split}
 \Gdim_{R}(M) & = \Gdim_{R/ \langle a \rangle}(M)+1 \\
 & \geq \edim\left(R/ \langle a \rangle\right)- \edim(S)+1 \\
 & = \edim(R) - \edim(S)
\end{split}
\end{equation*}

We next consider the general case. It is clear that $M\otimes_{S}\widehat{S}^{\mathfrak{n}}$ is a non-zero finitely generated $\widehat{S}^{\mathfrak{n}}$-module. Also, as we observed in the proof of Theorem \ref{3.5}, we have $\Gfd_{R}(M)= \Gfd_{\widehat{R}^{\mathfrak{m}}}\left(M\otimes_{S}\widehat{S}^{\mathfrak{n}}\right)$, $\edim(R)= \edim\left(\widehat{R}^{\mathfrak{m}}\right)$, $\edim(S)= \edim\left(\widehat{S}^{\mathfrak{n}}\right)$, and $\fd_{R}(S)= \fd_{\widehat{R}^{\mathfrak{m}}}\left(\widehat{S}^{\mathfrak{n}}\right)$. Therefore, we can assume that $R$ is $\mathfrak{m}$-adically complete and $S$ is $\mathfrak{n}$-adically complete. Then by Remark \ref{2.3}, $f$ has a Cohen factorization
$$R \xrightarrow{\dot{f}} R' \xrightarrow{f'} S$$
in which $\dot{f}$ is flat with regular closed fiber $R'/ \mathfrak{m}R'$, $R'$ is complete, and $f'$ is surjective. Besides, $\fd_{R'}(S)< \infty$. By Corollary \ref{3.3}, we have:
$$\Gfd_{R'}(M) \leq \Gfd_{R}(M) + \edim(R')-\edim(R)$$
Applying the special case to $f'$, we get:
\begin{equation*}
\begin{split}
 \Gfd_{R}(M) & \geq \Gfd_{R'}(M) + \edim(R) - \edim(R') \\
 & \geq \edim(R')- \edim(S) + \edim(R) - \edim(R') \\
 & = \edim(R)-\edim(S)
\end{split}
\end{equation*}
Therefore, we are done.
\end{proof}

\begin{corollary} \label{3.7}
Let $f:(R,\mathfrak{m}) \rightarrow (S,\mathfrak{n})$ be a local homomorphism of noetherian local rings. If $\fd_{R}(S)= \edim(R)- \edim(S)$, then $f$ is an exceptional complete intersection map.
\end{corollary}

\begin{proof}
Set $M=S$ in Corollary \ref{3.6} and note that $\Gfd_{R}(S)= \fd_{R}(S)= \edim(R)- \edim(S)$.
\end{proof}

\begin{remark} \label{3.8}
If $f:(R,\mathfrak{m}) \rightarrow (S,\mathfrak{n})$ is a complete intersection map, then by Remark \ref{2.3}, there is a Cohen factorization
\begin{equation*}
  \begin{tikzcd}
  R \arrow{r}{f} \arrow{d}[swap]{\dot{f}} & S \arrow{d}
  \\
  R' \arrow{r}{f'} & \widehat{S}^{\mathfrak{n}}
\end{tikzcd}
\end{equation*}

\noindent
of $f$ in which $\Ker(f')$ is generated by a regular sequence on $R'$. This implies that $\fd_{R'}\left(\widehat{S}^{\mathfrak{n}}\right)< \infty$, so \cite[3.3]{AFHe} yields $\fd_{R}(S)< \infty$. As a result, Theorem \ref{3.5} can in particular be applied to a complete intersection map to show that a certain bound on the Gorenstein flat dimension of a module along a complete intersection map can identify it as being exceptional.
\end{remark}

\begin{remark} \label{3.9}
Given a flat local homomorphism $\dot{f}:(R,\mathfrak{m})\rightarrow (R',\mathfrak{m}')$ of noetherian local rings with regular closed fiber $R'/\mathfrak{m}R'$, we ask the following question:
\begin{center}
If $R$ is G-regular, then is $R'$ G-regular?
\end{center}
This is a special case of a question of Takahashi; see \cite[Question 6.1]{Ta}. If the answer to this question is positive, then one can replace the finite flat dimension assumption on $f$ in Theorem \ref{3.5} with the assumption that $R$ is G-regular. Having done so, one can even show that $M$ is a free $S$-module. The argument goes as follows.

As in the proof of Theorem \ref{3.5}, we can reduce to the case where $R$ is $\mathfrak{m}$-adically complete and $S$ is $\mathfrak{n}$-adically complete. We just need to further note that $R$ is G-regular if and only if $\widehat{R}^{\mathfrak{m}}$ is G-regular by \cite[Corollary 4.7]{Ta}. Also, we have
$$\pd_{S}(M) = \fd_{S}(M) = \fd_{\widehat{S}^{\mathfrak{n}}}\left(M\otimes_{S}\widehat{S}^{\mathfrak{n}}\right) = \pd_{\widehat{S}^{\mathfrak{n}}}\left(M\otimes_{S}\widehat{S}^{\mathfrak{n}}\right),$$
which conspires with the fact that $S$ and $\widehat{S}^{\mathfrak{n}}$ are local to imply that $M$ is a free $S$-module if and only if $M\otimes_{S}\widehat{S}^{\mathfrak{n}}$ is a free $\widehat{S}^{\mathfrak{n}}$-module. Similarly, we can reduce to the case where $f$ is surjective. We just need to further note that given a Cohen factorization
$$R \xrightarrow{\dot{f}} R' \xrightarrow{f'} S$$
of $f$, $R'/ \mathfrak{m}R'$ is regular, so assuming a positive answer to the above question, we can conclude that $R'$ is G-regular. Finally, with $f$ being surjective, $M$ becomes a finitely generated $R$-module as well, so by Remark \ref{2.2}, $\Gfd_{R}(M) = \Gdim_{R}(M)$. On the other hand, $R$ is G-regular, so by Remark \ref{2.1}, $\Gdim_{R}(M) = \fd_{R}(M)$. Now the hypothesis reads $\fd_{R}(M) = \Gfd_{R}(M)\leq \edim(R)- \edim(S)$, so \cite[Theorem 3.1]{BIK} implies that $M$ is a free $S$-module, $f$ is an exceptional complete intersection map, and $\Gfd_{R}(M)= \fd_{R}(M)= \edim(R)- \edim(S)$.
\end{remark}

We next prove two analogues of Theorem \ref{3.5} involving Gorenstein injective dimension which are Theorems B and C from the introduction. We first need a lemma which might be of independent interest.

\begin{lemma} \label{3.10}
Let $f:(R,\mathfrak{m},k)\rightarrow (S,\mathfrak{n},l)$ be a homomorphism of noetherian local rings, $M$ an $S$-module, and $M^{\vee}= \Hom_{S}\left(M,E_{S}(l)\right)$ where $E_{S}(l)$ is the injective envelope of the $S$-module $l$. Suppose that either $R$ has a dualizing complex or $M$ is an artinian $S$-module. Then $\Gfd_{R}(M)=\Gid_{R}\left(M^{\vee}\right)$.
\end{lemma}

\begin{proof}
First suppose that $R$ has a dualizing complex, say $D$. Consider the Auslander class $\mathcal{A}_{D}(R)$ and the Bass class $\mathcal{B}_{D}(R)$ with respect to $D$; see \cite[Definitions 9.1 and 9.4]{CFH1}. Using the Hom-Tensor Adjunction and the injectivity of $E_{S}(l)$, We have
\begin{equation*}
\begin{split}
 H_{-i}\left(\RHom_{R}\left(D,M^{\vee}\right)\right) & = H_{-i}\left(\RHom_{R}\left(D,\Hom_{S}\left(M,E_{S}(l)\right)\right)\right) \\
 & = H_{-i}\left(\RHom_{R}\left(D,\RHom_{S}\left(M,E_{S}(l)\right)\right)\right) \\
 & \cong H_{-i}\left(\RHom_{S}\left(D\otimes_{R}^{\textrm{L}}M,E_{S}(l)\right)\right) \\
 & = H_{-i}\left(\Hom_{S}\left(D\otimes_{R}^{\textrm{L}}M,E_{S}(l)\right)\right) \\
 & \cong \Hom_{S}\left(H_{i}\left(D\otimes_{R}^{\textrm{L}}M\right),E_{S}(l)\right)
\end{split}
\end{equation*}
for every $i\in \mathbb{Z}$. As $E_{S}(l)$ is faithfully injective, the above display shows that for any $i\in \mathbb{Z}$, we have $H_{-i}\left(\RHom_{R}\left(D,M^{\vee}\right)\right)=0$ if and only if $H_{i}\left(D\otimes_{R}^{\textrm{L}}M\right)=0$. On the other hand, the diagram
\begin{equation*}
  \begin{tikzcd}[column sep=5em,row sep=2em]
  D\otimes_{R}^{\textrm{L}}\RHom_{S}\left(D\otimes_{R}^{\textrm{L}}M,E_{S}(l)\right) \arrow{r}{\simeq} \arrow{d}[swap]{\simeq} & D\otimes_{R}^{\textrm{L}}\RHom_{R}\left(D,\RHom_{S}\left(M,E_{S}(l)\right)\right) \arrow{d}{\delta_{M}} \arrow{d}{\delta_{M}}
  \\
  \RHom_{S}\left(\RHom_{R}\left(D,D\otimes_{R}^{\textrm{L}}M\right),E_{S}(l)\right) \arrow{r}{\RHom_{S}\left(\varepsilon_{M},E_{S}(l)\right)} & \RHom_{S}\left(M,E_{S}(l)\right) = M^{\vee}
\end{tikzcd}
\end{equation*}
in the derived category $\mathcal{D}(R)$, is commutative where $\varepsilon_{M}$ and $\delta_{M}$ are the natural morphisms in the definitions of the Auslander and Bass classes, respectively; see \cite[Definitions 9.1 and 9.4]{CFH1}. The above diagram shows that $\varepsilon_{M}$ is an isomorphism if and only if $\delta_{M}$ is an isomorphism. As a consequence, $M\in \mathcal{A}_{D}(R)$ if and only if $M^{\vee}\in \mathcal{B}_{D}(R)$. Now in light of \cite[Theorems 9.5 and 9.2]{CFH1}, we see that that $\Gfd_{R}(M) < \infty$ if and only if $\Gid_{R}\left(M^{\vee}\right) < \infty$. Hence we can assume that both quantities are simultaneously finite.

Let $I$ be an injective $R$-module, and $L$ a free resolution of $I$. Then we get
\begin{equation*}
\begin{split}
 \Tor_{i}^{R}(I,M)^{\vee} & = \Hom_{S}\left(H_{i}\left(L\otimes_{R}M\right),E_{S}(l)\right) \\
 & \cong H_{-i}\left(\Hom_{S}\left(L\otimes_{R}M,E_{S}(l)\right)\right) \\
 & \cong H_{-i}\left(\Hom_{R}\left(L,\Hom_{S}\left(M,E_{S}(l)\right)\right)\right) \\
 & = \Ext_{R}^{i}\left(I,M^{\vee}\right)
\end{split}
\end{equation*}
for every $i\in \mathbb{Z}$. Thus for any $i\in \mathbb{Z}$, we have $\Tor_{i}^{R}(I,M)=0$ if and only if $\Ext_{R}^{i}\left(I,M^{\vee}\right)=0$. By \cite[Theorems 2.22 and 3.14]{Ho}, we get:
\begin{equation*}
\begin{split}
 \Gfd_{R}(M) & = \sup\left\{i\geq 0 \suchthat \Tor_{i}^{R}(I,M) \neq 0 \textrm{ for some injective } R\textrm{-module } I\right\} \\
 & = \sup\left\{i\geq 0 \suchthat \Ext_{R}^{i}\left(I,M^{\vee}\right) \neq 0 \textrm{ for some injective } R\textrm{-module } I\right\} \\
 & = \Gid_{R}\left(M^{\vee}\right)
\end{split}
\end{equation*}
This proves the result in the first case.

Next assume that $M$ is an artinian $S$-module. Then $M$ is $\mathfrak{n}$-torsion, so it has an $\widehat{S}^{\mathfrak{n}}$-module structure that restricts to its original $S$-module structure via the completion map $S\rightarrow \widehat{S}^{\mathfrak{n}}$. More specifically, if $x\in M$ and $(a_{n}+\mathfrak{n}^{n})_{n\geq 1}\in \widehat{S}^{\mathfrak{n}}$, then since $M$ is $\mathfrak{n}$-torsion, there is a $t\geq 1$ such that $\mathfrak{n}^{t}x=0$, so $(a_{n}+\mathfrak{n}^{n})_{n\geq 1}x:= a_{t}x$ defines the desired $\widehat{S}^{\mathfrak{n}}$-module structure on $M$. In addition, one has $M\cong M\otimes_{S}\widehat{S}^{\mathfrak{n}}$ both as $S$-modules and $\widehat{S}^{\mathfrak{n}}$-modules; see \cite[Proposition 2.1.15 and Corollary 2.2.6]{SS}. As $f$ is local, $\mathfrak{m}S\subseteq \mathfrak{n}$, so $M$ is $\mathfrak{m}$-torsion as well, thereby it has an $\widehat{R}^{\mathfrak{m}}$-module structure which is restricted from its $\widehat{S}^{\mathfrak{n}}$-module structure via $\widehat{f}: \widehat{R}^{\mathfrak{m}}\rightarrow \widehat{S}^{\mathfrak{n}}$. It follows that $M\cong M\otimes_{R}\widehat{R}^{\mathfrak{m}}$ as $\widehat{S}^{\mathfrak{n}}$-modules. We thus have the following isomorphisms in the derived category $\mathcal{D}\left(\widehat{R}^{\mathfrak{m}}\right)$:
\begin{equation*}
\begin{split}
 \Hom_{\widehat{S}^{\mathfrak{n}}}\left(M,E_{\widehat{S}^{\mathfrak{n}}}\left(\widehat{S}^{\mathfrak{n}}/ \mathfrak{n}\widehat{S}^{\mathfrak{n}}\right)\right) & \simeq \RHom_{\widehat{S}^{\mathfrak{n}}}\left(M\otimes_{R}^{\textrm{L}}\widehat{R}^{\mathfrak{m}},E_{\widehat{S}^{\mathfrak{n}}}\left(\widehat{S}^{\mathfrak{n}}/ \mathfrak{n}\widehat{S}^{\mathfrak{n}}\right)\right) \\
 & \simeq \RHom_{R}\left(\widehat{R}^{\mathfrak{m}}, \RHom_{\widehat{S}^{\mathfrak{n}}}\left(M,E_{\widehat{S}^{\mathfrak{n}}}\left(\widehat{S}^{\mathfrak{n}}/ \mathfrak{n}\widehat{S}^{\mathfrak{n}}\right)\right)\right) \\
 & \simeq \RHom_{R}\left(\widehat{R}^{\mathfrak{m}}, \RHom_{\widehat{S}^{\mathfrak{n}}}\left(M\otimes_{S}^{\textrm{L}}\widehat{S}^{\mathfrak{n}},E_{\widehat{S}^{\mathfrak{n}}}\left(\widehat{S}^{\mathfrak{n}}/ \mathfrak{n}\widehat{S}^{\mathfrak{n}}\right)\right)\right) \\
 & \simeq \RHom_{R}\left(\widehat{R}^{\mathfrak{m}}, \RHom_{S}\left(M,\RHom_{\widehat{S}^{\mathfrak{n}}}\left(\widehat{S}^{\mathfrak{n}},E_{\widehat{S}^{\mathfrak{n}}}\left(\widehat{S}^{\mathfrak{n}}/ \mathfrak{n}\widehat{S}^{\mathfrak{n}}\right)\right)\right)\right) \\
 & \simeq \RHom_{R}\left(\widehat{R}^{\mathfrak{m}}, \RHom_{S}\left(M,E_{S}(l)\right)\right) \\
 & = \RHom_{R}\left(\widehat{R}^{\mathfrak{m}},M^{\vee}\right)
\end{split}
\end{equation*}

Note that $\widehat{R}^{\mathfrak{m}}$ is $\mathfrak{m}$-adically complete, so it has a dualizing complex. Thus in view of \cite[Theorems 4.27 and 9.11]{CFH1}, the previous case, and the above isomorphisms, we achieve:
\begin{equation*}
\begin{split}
 \Gfd_{R}(M) & = \Gfd_{\widehat{R}^{\mathfrak{m}}}\left(M\otimes_{R}\widehat{R}^{\mathfrak{m}}\right) \\
 & = \Gfd_{\widehat{R}^{\mathfrak{m}}}\left(M\right) \\
 & = \Gid_{\widehat{R}^{\mathfrak{m}}}\left(\Hom_{\widehat{S}^{\mathfrak{n}}}\left(M,E_{\widehat{S}^{\mathfrak{n}}} \left(\widehat{S}^{\mathfrak{n}}/ \mathfrak{n}\widehat{S}^{\mathfrak{n}}\right)\right)\right) \\
 & = \Gid_{\widehat{R}^{\mathfrak{m}}}\left(\RHom_{R}\left(\widehat{R}^{\mathfrak{m}},M^{\vee}\right)\right) \\
 & = \Gid_{R}\left(M^{\vee}\right)
\end{split}
\end{equation*}
This proves the result in the second case.
\end{proof}

\begin{theorem} \label{3.11}
Let $f:(R,\mathfrak{m},k) \rightarrow (S,\mathfrak{n},l)$ be a local homomorphism of noetherian local rings with $\fd_{R}(S)<\infty$ in which $R$ has a dualizing complex and $S$ is $\mathfrak{n}$-adically complete, and $M$ a non-zero artinian $S$-module with $\Gid_{R}(M)\leq \edim(R)- \edim(S)$. Then the following assertions hold:
\begin{enumerate}
\item[(i)] $M$ is a Gorenstein injective $S$-module.
\item[(ii)] $f$ is an exceptional complete intersection map.
\item[(iii)] $\Gid_{R}(M)= \edim(R)- \edim(S)$.
\end{enumerate}
In particular, if $M$ is Gorenstein injective as an $R$-module, then $\edim(R)= \edim(S)$.
\end{theorem}

\begin{proof}
Write $(-)^{\vee}=\Hom_{S}\left(-,E_{S}(l)\right)$. Matlis duality theory implies that $M$ is a Matlis reflexive $S$-modules, so $M\cong M^{\vee\vee}$, and $M^{\vee}$ is a non-zero finitely generated $S$-module. By Lemma \ref{3.10}, we have
$$\Gfd_{R}\left(M^{\vee}\right) = \Gid_{R}\left(M^{\vee\vee}\right) = \Gid_{R}(M) \leq \edim(R)- \edim(S).$$
Therefore, Theorem \ref{3.5} implies that $M^{\vee}$ is a totally reflexive $S$-module, $f$ is an exceptional complete intersection map, and
$$\Gid_{R}(M)= \Gid_{R}\left(M^{\vee\vee}\right) = \Gfd_{R}\left(M^{\vee}\right)= \edim(R)- \edim(S).$$
But $S$ is $\mathfrak{n}$-adically complete, so it has a dualizing complex, thereby another application of Lemma \ref{3.10} yields
$$\Gid_{S}(M)= \Gid_{S}\left(M^{\vee\vee}\right) = \Gfd_{S}\left(M^{\vee}\right)= \Gdim_{S}\left(M^{\vee}\right)=0,$$
so $M$ is a Gorenstein injective $S$-module.
\end{proof}

\begin{lemma} \label{3.12}
Let $(R,\mathfrak{m},k)$ be a noetherian local ring, and $M$ a non-zero finitely generated or artinian $R$-module. Let $E_{R}(k)$ denote the injective envelope of $k$, and $\type_{R}(M)= \rank_{k}\left(\Ext_{R}^{\depth_{R}(\mathfrak{m},M)}(k,M)\right)$. Then $M$ is an injective $R$-module if and only if $M\cong E_{R}(k)^{\type_{R}(M)}$.
\end{lemma}

\begin{proof}
See \cite[Lemma 3.1]{Fa}
\end{proof}

\begin{theorem} \label{3.13}
Let $f:(R,\mathfrak{m},k) \rightarrow (S,\mathfrak{n},l)$ be a local homomorphism of noetherian local rings with $\fd_{R}(S)<\infty$ in which $R$ has a dualizing complex and $S$ is G-regular, and $M$ a non-zero artinian $S$-module with $\Gid_{R}(M)\leq \edim(R)- \edim(S)$. Then the following assertions hold:
\begin{enumerate}
\item[(i)] $M$ is an injective $S$-module.
\item[(ii)] $f$ is an exceptional complete intersection map.
\item[(iii)] $\Gid_{R}(M)= \edim(R)- \edim(S)$.
\end{enumerate}
In particular, if $M$ is Gorenstein injective as an $R$-module, then $\edim(R)= \edim(S)$.
\end{theorem}

\begin{proof}
We reduce to the case where $S$ is $\mathfrak{n}$-adically complete. Firstly, $S$ is G-regular if and only if $\widehat{S}^{\mathfrak{n}}$ is G-regular; see \cite[Corollary 4.7]{Ta}. We next note that $M$ is an artinian $S$-module, so as we observed in the proof of Lemma \ref{3.10}, it has an $\widehat{S}^{\mathfrak{n}}$-module structure that restricts to its original $S$-module structure via the completion map $S\rightarrow \widehat{S}^{\mathfrak{n}}$, and $M\cong M\otimes_{S}\widehat{S}^{\mathfrak{n}}$ both as $S$-modules and $\widehat{S}^{\mathfrak{n}}$-modules. Also, $M$ is an artinian $\widehat{S}^{\mathfrak{n}}$-module. We further note that $\edim\left(\widehat{S}^{\mathfrak{n}}\right)= \edim(S)$ and $E_{\widehat{S}^{\mathfrak{n}}}\left(\widehat{S}^{\mathfrak{n}}/ \mathfrak{n}\widehat{S}^{\mathfrak{n}}\right) \cong E_{S}(S/ \mathfrak{n}) \cong E_{S}(l)$. Since $M$ is an artinian $S$-module, $\Ass_{S}(M)=\{\mathfrak{n}\}$, so $\depth_{S}(\mathfrak{n},M)=0$, whence $\type_{S}(M)= \rank_{l}\left(\Soc_{S}(M)\right)$. The same holds when we consider $M$ as an $\widehat{S}^{\mathfrak{n}}$-modules. But socle is the sum of simple submodules, and the $S$-submodules and $\widehat{S}^{\mathfrak{n}}$-submodules of $M$ coincide, so we conclude that $\Soc_{\widehat{S}^{\mathfrak{n}}}(M) = \Soc_{S}(M)$, thereby $\type_{\widehat{S}^{\mathfrak{n}}}(M) = \type_{S}(M)$. By Lemma \ref{3.12}, $M$ is an injective $S$-module if and only if $M\cong E_{S}(l)^{\type_{S}(M)}$. Similarly, $M$ is an injective $\widehat{S}^{\mathfrak{n}}$-module if and only if $M\cong E_{\widehat{S}^{\mathfrak{n}}}\left(\widehat{S}^{\mathfrak{n}}/ \mathfrak{n}\widehat{S}^{\mathfrak{n}}\right)^{\type_{\widehat{S}^{\mathfrak{n}}}(M)}$. It follows that $M$ is an injective $S$-module if and only if $M$ is an injective $\widehat{S}^{\mathfrak{n}}$-module. Finally, by \cite[Corollary 4.2 (b)(F)]{AF1}, we have $\fd_{R}\left(\widehat{S}^{\mathfrak{n}}\right)\leq \fd_{R}(S)< \infty$. Therefore, we can replace $S$ by $\widehat{S}^{\mathfrak{n}}$ and assume that $S$ is $\mathfrak{n}$-adically complete.

Write $(-)^{\vee}=\Hom_{S}\left(-,E_{S}(l)\right)$. Matlis duality theory implies that $M$ is a Matlis reflexive $S$-modules, so $M\cong M^{\vee\vee}$, and $M^{\vee}$ is a non-zero finitely generated $S$-module. By Lemma \ref{3.10}, we have
$$\Gfd_{R}\left(M^{\vee}\right) = \Gid_{R}\left(M^{\vee\vee}\right) = \Gid_{R}(M) \leq \edim(R)- \edim(S).$$
Therefore, Theorem \ref{3.5} implies that $M^{\vee}$ is a totally reflexive $S$-module, $f$ is an exceptional complete intersection map, and
$$\Gid_{R}(M)= \Gid_{R}\left(M^{\vee\vee}\right) = \Gfd_{R}\left(M^{\vee}\right)= \edim(R)- \edim(S).$$
But $S$ is G-regular, so $M^{\vee}$ is a free $S$-module, say $M^{\vee}\cong S^{n}$ for some $n\geq 1$. Hence $M\cong M^{\vee\vee} \cong E_{S}(l)^{n}$ is an injective $S$-module.
\end{proof}

\begin{remark} \label{3.14}
One should note that Theorems \ref{3.11} and \ref{3.13} can be specialized to a complete intersection map as in Remark \ref{3.8}.
\end{remark}

\section*{Acknowledgement}

It is a pleasure to express my sincerest gratitude to professors Srikanth Iyengar and Lars Christensen for their invaluable comments and suggestions on this manuscript which is part of my Ph.D. thesis at Clemson University.


\end{document}